\documentclass[11pt, reqno]{article}
\usepackage{amssymb, amsmath, amsthm, amsrefs}
\usepackage{latexsym, color, esint}
\usepackage{graphicx, subcaption}

\def\N{{\mathbb N}} 
\def\R{{\mathbb R}} 

\newtheorem{theorem}{Theorem}[section]

\newtheorem{definition}[theorem]{Definition}
\newtheorem{proposition}[theorem]{Proposition}
\newtheorem{remark}[theorem]{Remark}

\def\ds{\displaystyle}

\def\grad{\nabla}

\def\E{\mathbf{E}}

\newcommand{\pd}[2]{\frac{\partial #1}{\partial #2}}
\def\Pr{{\mathrm P}} 
\def\P{{\mathbb P}}

\def\H{{\mathcal H}}

\def\Dom{\mathrm{Dom}}
\newcommand{\tr}[1]{\mathrm{tr}(#1)}

\def\dist{\mathrm{dist}}

\DeclareUnicodeCharacter{02BC}{\textquotesingle}

\title{The Existence of Anomalous Dissipation over Bounded Interior Domains}
\author{Ethan Dudley\footnotemark[2]\; and Konstantina Trivisa\footnotemark[2]\;\footnotemark[1]}

\begin{document}
\maketitle
\renewcommand{\thefootnote}{\fnsymbol{footnote}}
\footnotetext[2]{ Mathematics Department, University of Maryland College Park, Maryland, United States\newline}
\footnotetext[1]{ Institute for Physical Science and Technology, University of Maryland College Park, Maryland, United States\newline
Correspondence may be sent to edudley1@terpmail.umd.edu}
\begin{abstract}
   A prevalent feature of three-dimensional turbulence is the presence of anomalous dissipation, or that the mean rate of energy dissipation is bounded below by a positive number in the inviscid limit. This is thought to be due to the nonlinear convection term in the Navier Stokes equations stretching vortex tubules and thereby increasing the amount of small scale oscillations within the flow. In this paper, we construct an example of a linear Stokes flow within a sphere that exhibits anomalous dissipation. 
\end{abstract}

\tableofcontents
\section{Introduction}

When considering the inviscid limit problem for the Navier Stokes equations, an important consideration one has to account for is the underlying energy budget of the problem. The Navier Stokes equations are a dissipative system, while the Euler equations (the weak inviscid limit of the Navier Stokes equations) are formally conservative. However, despite this numerous physical experiments \cites{touil2002decay, pearson2002measurements,  sreenivasan1984scaling, frisch1995turbulence} show that as the viscosity vanishes a certain amount of mean viscous dissipation remains. It is commonly thought that this behavior is due to the convective acceleration term (i.e. $u^\nu \cdot \grad u^\nu$) stretches the vortex tubules within the flow  which in turn increases the total vorticity (and small scale oscillations) within the flow \cite{onsager1949statistical}. 

If $\nu > 0$ is the viscosity and $u^\nu$ is the associated viscous fluid velocity over a domain $D \subset \R^3$, then we traditionally say that the flow exhibits anomalous dissipation if there exists $\varepsilon_0 > 0$ such that
\begin{equation}\label{eqn: anomalous dissipation}  
    \limsup_{\nu \to 0}\nu\E\int_0^t\|\grad u^\nu\|_{L^2(D)}^2 = \varepsilon_0 > 0.
\end{equation}
This definition of (global) anomalous dissipation \cite{frisch1995turbulence} can be readily extended to a local type definition over compact subsets $K \subset D$ and is easily computable within a numerical simulation meaning it is used extensively within the literature. An interesting mathematical question, is how much regularity does $u^\nu$ need to retain in its inviscid limit for \eqref{eqn: anomalous dissipation} to hold. Onsager conjectured that if $u^\nu \in [C(0,T, C^{\gamma}(D))]^3$ with $\gamma \leq \frac{1}{3}$ then $u^\nu$ exhibits anomalous dissipation. Onsager's conjecture is supported by the fact that there exist Euler flows over unbounded domains with Holder regularity $\gamma \leq \frac{1}{3}$ which are dissipative \cite{isett2018proof} and even in the larger space of Besov regular solutions \cite{novack2023intermittent}. Moreover, recently Armstrong and Vicol \cite{armstrong2023anomalous} were able to construct a highly oscillatory flow field by gluing together shear flows together at infinitely many length scales such that a passive scalar subject to this flow field exhibits anomalous dissipation. Based on this approach, Brue and de Lellis \cite{brue2023anomalous} were able to construct 2.5 dimensional solutions (when $u^\nu$ is a 3 dimensional vector field but only depends on 2 spatial input variables) to the Navier Stokes equations which exhibit \eqref{eqn: anomalous dissipation} by setting the $z$-velocity to be a passive scalar type quantity. Despite this great progress, the case of a fully three-dimensional analytical velocity field to the Navier Stokes equations which exhibits \eqref{eqn: anomalous dissipation} remains open to the authors knowledge.

In this work, we wish to address how the presence of a boundary affects the existence of anomalous dissipation. Typically, the presence of a boundary makes most classical techniques harder to use, for instance Calderon-Zygmund estimates for the pressure fail to hold due to the boundary, and in the inviscid limit a boundary layer will form \cites{prandtl1928motion, sammartino1998zero1, kato1972nonstationary}. In the case of a no-slip boundary condition, Kato was able to show that for a highly regular domain and $[C^1((0,T) \times D)]^3$ solutions then there is \textit{no} anomalous dissipation if and only if the amount of vorticity within a thin region next to the boundary vanishes \cite{kato1972nonstationary}. Kato's proof uses a boundary correction type approach and there have been many variants of Kato's criterion for the lack of anomalous dissipation. In particular, we refer to the result by Drivas and Nguyen \cite{drivas2018onsager} which uses a mollifier type approach and breaks the flow into the bulk and boundary components to show the regularity assumptions required on each part. In that work the authors were able to show that if the bulk flow is uniformly $[L^3(0,T, B_{3, c_0}^\sigma(D)]^3$ where $\sigma \in [\frac{1}{3},1]$ and the boundary flow is $[L^\infty((0,T) \times D)]^3$ then there is no anomalous dissipation. However, the authors in \cite{drivas2018onsager} work with an unforced flow which would almost never appear in nature. Motivated in part by their work, we work with the Navier-slip condition instead of the no-slip condition and we will study flows which are only $L^2((0,T) \cap H^s(D))$ near the boundary with $s \in (0, \frac{1}{2})$. This is substantially less control over how the flows behave at the boundary and we will show that we get a non-zero limit in \eqref{eqn: anomalous dissipation} which is the opposite result from \cite{drivas2018onsager}. Moreover our example should be able to extended to show an non-vanishing limit over any compactly supported subset in the domain by taking a properly weighted sum of external forces which concentrate energy on concentric shells with radii from a countable dense set in $\R_+$ which shows that one sufficient condition for \eqref{eqn: anomalous dissipation} is to concentrate energy over sets of Lebesgue measure 0 within the domain. Furthermore, contrary to most comments within articles for unbounded domains, our example uses a fixed external forcing which is not highly irregular (it is in fact $C^\infty$ within the interior of the domain) but rather just concentrates energy at the boundary. This is highly analogous to how Stokes 2nd problem which is not ``anomalous" which shows that \eqref{eqn: anomalous dissipation} is not a good way of measuring the contribution of the nonlinear term to the acceleration of viscous dissipation in the inviscid limit.

Results of this kind are known already in the context of stationary solutions to the stochastically forced Navier Stokes equations over unbounded domains. For instance, Bedrossian et.al. \cite{bedrossian2019sufficient} showed that \eqref{eqn: anomalous dissipation} is not a good definition for anomalous dissipation in the case of statistically stationary solutions, because stationary solutions to the stochastic heat equation subject to a zero drift, white-in-time colored-in-space noise (which is independent of $\nu$) will satisfy the non-vanishing viscous dissipation assumption due to its energy balance:
\[
    \begin{cases}
        dw = \nu \Delta w dt + gdW_t\\
        w(0) = w_0
    \end{cases} \Rightarrow \nu\E\int_0^t\|\grad w\|_{L^2(D)}^2 = \frac{1}{2}\E\int_0^t\|g\|_{L^2(D)}^2 \neq 0.
\]
As such in the case of stochastic flows, authors have tried finding alternative definitions for anomalous dissipation which are directly related to how the convection term contributes to the energy balance. See \cites{bedrossian2019sufficient, papathanasiou2021sufficient, dudley2024necessary, drivas2022self}. However to the authors knowledge, the example we outline here is the first such example in the for non-statistically stationary solutions. 
In particular, we show that in the presence of a boundary, we can build non-statistically stationary solutions to the linear Stokes equation which still satisfy \eqref{eqn: anomalous dissipation}. This analysis also works for the linear heat equation, but we will work with the Stokes operator as it is closely tied to the Navier Stokes equation which is the motivating system for this problem. Heuristically, one can argue that our approach still increases the total vorticity within the flow --- without any vortex stretching --- as the boundary is a natural source for generating vorticity within the flow.

\subsection{The Physical Implications of Including a Boundary}
In this work, we take the domain $D = B(0, R) \subset \R^3$ where $R \in (0, \infty)$ is a fixed value. We will further assume that this boundary is impermeable to the fluid, but the total shear stress induced on the fluid is proportional to the momentum of the fluid at the wall. Consider the Stokes problem associated with these boundary conditions:
\begin{equation}\label{eqn: Stokes w/ slip}
    \begin{cases}
        dz^\nu = \big(\nu \Delta z^\nu  + \grad p^\nu + f\big)dt + gdW_t & x \in D,\; t > 0\\
        \grad \cdot z^\nu = 0 & x \in D,\; t > 0\\
        z^\nu \cdot n^{\partial D} = 0 & x \in \partial D,\; t > 0\\
        n^{\partial D} \cdot \grad z_\tau^\nu + \alpha z_\tau^\nu = 0 &x \in \partial D,\; t > 0\\
        z^\nu(0) = z_0 & x \in D.
    \end{cases}
\end{equation}
Here $\alpha \in L^\infty(\partial D)$ is called the slip length, $n^{\partial D}$ is the outward unit normal vector to the boundary $\partial D$, and $z_\tau^\nu|_{\partial D} = z^\nu|_{\partial D} - (n^{\partial D} \cdot z^\nu)n^{\partial D}|_{\partial D}$ is the trace of $z^\nu$ in the tangential direction along the boundary. 

The slip-type boundary condition proposed here was first considered by Navier \cite{navier1822memoire}. Here $\alpha \geq 0$ is a ``constant" of proportionality that balances the friction the fluid experiences along the boundary with the acceleration due to the pressure gradient. Alternatively, one can consider the slip condition as a growth rate for the tangential component of the vorticity generated at the wall (after accounting for the curvature of the wall as well) \cite{chen2010study}. In particular, this friction-type boundary condition has been derived from the kinetic theory of homogenization in \cites{bardos2006half, coron1988classification} and has been justified as the effective boundary condition for flows over rough boundary \cites{gerard2010relevance, mikelic2000interface}. Furthermore, experimental evidence suggests that at sufficiently large Reynolds numbers or in domains with curvature, the no-slip boundary condition fails to capture important information about the flow \cites{zhu2002limits, einzel1990boundary, jager2003couette}. As such, the study of the Navier Stokes equations with Navier Slip boundary condition has become more prevalent within the literature, see \cites{chen2010study, neustupa2018regularity, xiao2007vanishing} and the references within for a sample. Here we take a more simplified model \eqref{eqn: Stokes w/ slip} which linearize the problem and directly related to the Navier Stokes system, see Section \ref{sec: martingale solutions} for more information. 

Another important consideration for why the slip boundary was chosen is that if $\nu$ is fixed, then as $\alpha \to \infty$ one can recover the non-slip condition and as $\alpha \to 0$ we recover the full slip (also known as the zero-flux) boundary condition; marking the Navier-slip condition as a generalization of the other types of physical boundary conditions typically encountered in experiments \cite{chen2010study}. In this work, the Navier-slip condition is used not only to account for physical relevance but also as a measure to ensure the kinetic energy at the boundary is non-zero. As we will show in the following sections, lacking control over the kinetic energy at the boundary is a sufficient condition to show the existence of anomalous dissipation over the entire domain even if there is no convective term to mix length scales together. 

\subsection{Outline for the Article}
First in Section \ref{sec: martingale solutions} we review the existance of solutions to the Stokes problem and how one can construct solutions to the Navier Stokes equations using those solutions to the Stokes problem. Then in Section \ref{sec: exist anomalous dissipation} we choose our forcing such that our family of non-statistically stationary solutions to the Stokes problems satisfy \eqref{eqn: anomalous dissipation}. Finally, in Section \ref{sec: numerical simulations} we numerically simulation our example using finite differences to illustrate how the family of constructed solutions from the previous section behave in both the stochastic and deterministic setting.  

\subsection{Notation}
Let $X$ be any Banach space containing scalar fields $h:\R^d \to \R$ subject to some requirement. We will denote the vector space equivalent as $[X]^d = \{h \mid h_i \in X, \; i \in \{1,2,\dots, d \}\}$ for $d = 2,3$. For $x \in \R^d$ and $R > 0$ we write $B(x,R) \subset \R^d$ for the ball centered at $x$ with a radius of $R$. Let $C_c^\infty(D)$ be the space of continuously differentiable functions compactly supported within $D$. And when $D$ has a boundary, we will use $\H^{d-1}$ to express the surface measure of $\partial D$.

For $s \in (0,1)$ and define the following spaces
\begin{align*}
    [L_\sigma^2(D)]^d &= \{f \in [L^2(D)]^d \mid \grad \cdot f = 0\}\\
    [H_{\sigma, \tau}^1(D)]^d &= \{f \in [H^1(D)]^d \mid \grad \cdot f = 0,\quad n^{\partial D}\cdot \grad f_\tau + \alpha f_\tau = 0|_{\partial D}\}\\
    \dot{H}^s(D) &= \{f \in L^2(D) \mid \|f\|_{\dot{H}^s}^2 := \int_D \int_D \frac{|f(x) - f(y)|^2}{|x-y|^{3+2s}}\;dxdy < \infty\}.
\end{align*}
Each space will be equipped with its induced (semi-) norm. Without loss of generality, the (semi-) norm over $X$ and $[X]^3$ will be denoted by $\|\cdot\|_X$ for simplicity. Here $f_\tau$ is the tangential component of $f$ to the boundary and $n^{\partial D} \cdot \grad f_\tau + \alpha f_\tau = 0|_{\partial D}$ holds in the sense of distributions.

For convenience we will use the Einstein summation convention: $a_jb_j = \sum_{j=1}^da_jb_j$. 
Moreover, for two given rank-two tensors $A$ and $B$ we define the Frobenius product and norm as $A:B = A_{ij}B_{ij} = \sum_{i}\sum_j A_{ij}B_{ij}$ and $|A| = \sqrt{A:A}$.

Finally, we will commonly use $C$ to denote a constant independent of $\nu, z^\nu$, but may depend on $s, D, \partial D$ or $T$. Furthermore, without loss of generality, we will use $C$ repeatedly across inequalities even if the constant of the coefficient increases/decreases. 

\section{Existence of Solutions}\label{sec: martingale solutions}
In order to construct our family of solution processes $\{z^\nu\}_{\nu > 0}$ which satisfy both the linear Stokes problem as well as \eqref{eqn: anomalous dissipation} we will concentrate the external forcing near the boundary. This has the affect of speeding up the fluid velocity at the wall causing larger shearing affects (due to the slip boundary condition) which then moves into the interior as small scale oscillations.  Before getting into the specific details to ensure the family of solutions exhibit anomalous dissipation, we first recall a few results about the linearized Stokes problem.

\subsection{The Stokes Operator with Slip conditions}\label{subsec: Stokes operator} 
To construct solutions $z^\nu$ to \eqref{eqn: Stokes w/ slip} we will use a semigroup method using the techniques of Da Prato \cite{da2014stochastic}. \newline

Let $\P:L^2(D) \to L_\sigma^2(D)$ be the Leray projection and $\Delta_{NS}:H^2(D) \cap H_{\tau, \sigma}^1(D) \to L_\tau^2(D)$ be the Laplacian subject to the Navier-slip boundary condition:
\begin{align*}
    \begin{cases}
        \Delta_{NS}\phi = 0 & x \in D\\
        n^{\partial D}\cdot \grad \phi_\tau + \alpha \phi_\tau = 0 & x \in \partial D
    \end{cases}
\end{align*}
The Stokes operator with Navier-slip conditions is defined as $A_{NS} = \Pr(-\Delta_{NS})$. It is known that $A_{NS}$ is a positive, self-adjoint, unbounded, linear operator, and with domain $\Dom(A_{NS}) = H^2(D) \cap H_{\tau, \sigma}^1(D)$. Furthermore $A_{NS}$ has compact resolvent and by the Spectral Theorem the eigenvalues of $A_{NS}$ are such that $0 < \lambda_1 \leq \lambda_2 \leq \dots$ and the associated eigenfunctions $\{q_j\}_{j\in \N}$ form an orthonormal basis in $L_{\sigma, \tau}^2(D)$ \cites{chen2010study, Amrouche2021}. Then the noise can be represented in $L^\infty(0,T, L^2(D))$ $\P$-a.s. as 
\[
    gdW_t = \sum_{j=1}^\infty \langle g, q_j\rangle_{L^2(D)}q_jd\beta_j(t)
\]
with the cross-variation $\int_0^T[gdW_t, gdW_t] = \int_0^T\|g\|_{L^2(D)}^2 < \infty$. Here $\langle \cdot , \cdot \rangle_{L^2(D)}$ is the $[L^2(D)]^3$ inner product, and $\{\beta_j\}_{j \in \N}$ is a family of i.i.d. 1-dimensional Standard Brownian motions. See \cite{da2014stochastic} for more information regarding the construction of the noise. Also using integration by parts we can define
\[
    \|A_{NS}^{1/2}h\|_{L^2(D)}^2 := \|\grad h\|_{L^2(D)}^2 + \|\sqrt{\alpha}\;\tr{h}\|_{L^2(\partial D)}^2
\]
for all $h \in \Dom(A_{NS}^{1/2}) \subset H^1(D)$, where $\mathrm{tr}: H^1(D) \to L^2(\partial D)$ is the trace operator. 

Now we can recast \eqref{eqn: Stokes w/ slip} using the Stokes operator $A_{NS}$ as a general linear parabolic problem: 
\begin{equation}\label{eqn: general linear stokes problem}
    \begin{cases}
        dz^\nu = \big(-\nu A_{NS}z^\nu  + f\big)dt + gdW_t & x \in D,\; t > 0\\
        z^\nu(0) = z_0 & x \in D.
    \end{cases}
\end{equation}
It is known that $-A_{NS}$ generates an analytic semigroup provided $D$ is $C^{1,1}$\cite{ter2020holder}. As such define:
\[
    z^\nu(t) = e^{-\nu tA_{NS}}z_0 + \int_0^t e^{-\nu(t-s)A_{NS}}f\;ds + \int_0^t e^{-\nu(t-s)A_{NS}}g\;dW_s. 
\]
Then $z^\nu$ is the unique weak solution to \eqref{eqn: general linear stokes problem} \cite{da2014stochastic}, and by both Ito's formula and Young's inquality $z^\nu$ satisfies the following energy inequality
\begin{align}\label{eqn: Stokes energy uniform bounds}
    \frac{1}{2}\E\big(\sup_{t \in [0,T]}\|z^\nu(t)\|_{L^2(D)}^2\big) +  &2\nu\E\int_0^T\|A_{NS}^{1/2}z^\nu\|_{L^2(D)}^2\\
    &\leq \E\|z_0\|_{L^2(D)}^2 + \E\int_0^T\|f\|_{L^2(D)}^2 + \E\int_0^T\|g\|_{L^2(D)}^2.\nonumber
\end{align}
Thus the trajectories of $z^\nu$ are in $[L^\infty(0,T, L_\sigma^2(D))]^3 \cap [L^2(0,T, H_{\sigma, \tau}^1(D))]^3$ $\P$-a.s. when $\nu > 0$. It is possible to show higher classes of regularity by appealing to the analyticity of the semigroup operator, but such results are unneeded in this work. 

An important class of solutions are known as \textit{Statistically stationary} solutions:
\begin{definition}
    We say that $z^\nu$ is statistically stationary, if $z^\nu$ is the unique weak solution to \eqref{eqn: general linear stokes problem} and if for all $t > 0$
    $z^\nu(\cdot + t) = z^\nu(\cdot)$ in law. 
\end{definition}
\begin{remark}
    An important aspect about statistically stationary solutions is that on average the mean kinetic energy remains constant. This means that $\E\|z^\nu(t)\|_{L^2(D)}^2 = \E\|z_0\|_{L^2(D)}^2$.
\end{remark}

\subsection{The Nonlinear Problem}\label{subsec: the nonlinear problem}
If we want to construct solutions to the Navier Stokes equation, note that we can do so by including a random differential equation which accounts for how the convection term mixes the external forcing across length scales. To do this, we say that for each fixed $\omega \in \Omega$ $z^\nu(\omega)$ is now a deterministic solution to the linear Stokes problem and we let $v^\nu$ be a weak solution to 
\begin{equation}\label{eqn: nonlinear part}
    \begin{cases}
        dv^\nu = \big(-\nu \Pr(-\Delta_{NS})v^\nu - \Pr((v^\nu + z^\nu(\omega) \cdot \grad (v^\nu + z^\nu(\omega))\big)dt & x \in D,\; t > 0\\
        \grad \cdot v^\nu = 0& x \in D,\; t > 0\\
        n^{\partial D} \cdot v^\nu = 0& x \in \partial D,\; t > 0\\
         n^{\partial D} \cdot \grad v_\tau^\nu + \alpha v_\tau^\nu = 0& x \in \partial D,\; t > 0\\
        v^\nu(0) = u_0 & x \in D.
    \end{cases}
\end{equation}
Showing the existence of $v^\nu$ which is a weak solution to \eqref{eqn: nonlinear part} is straight forward to establish; one can use the eigenfunctions of the Stokes operator $A_{NS}$ in a Faedo-Galerkin scheme similar to the method used in \cite{flandoli1995martingale}.
We remark that \eqref{eqn: nonlinear part} is entirely deterministic (for fixed $\omega$) and if we let $u^\nu(\omega) = v^\nu + z^\nu(\omega)$ then $u^\nu(\omega)$ is a weak solution to the incompressible Navier Stokes equations subject to the Navier slip condition. However, this holds for each fixed realization, in order to generalize this for all possible realizations --- if we want $u^\nu$ to $\P$-a.s. solve the incompressible Navier Stokes equations--- we can still establish existence of solutions but the stochastic basis associated to the solution will not be known apriori \cite{flandoli1995martingale}. 

\begin{remark}
    Here the convection term $u^\nu \cdot \grad u^\nu$ shows the mixing of the (dissipative) linear problem across different length scales within the interior of the domain. 
\end{remark}
\begin{remark}
    The property that $v^\nu$ is a suitable weak solution to \eqref{eqn: nonlinear part} so that $u^\nu = v^\nu + z^\nu$ is a weak solution to Navier Stokes equations seems to depend on the solution $z^\nu$ to the linear problem. However this is not true. See Theorem 2.8 in \cite{romito2010existence} for the proof of such a result. As such when $g \equiv 0$ there is no difference between suitable weak solutions in the classical deterministic sense of Caffarelli, Kohn and Nirenberg \cite{caffarelli1982partial} and the stochastic case.
\end{remark}
\begin{remark}\label{remark: navier stokes anomalous dissipation}
    In the next subsection we choose the external forcing $f,g$ which are concentrated at the boundary to ensure that 
    \[
        0 < \varepsilon_0 \leq \lim_{\nu \to 0}\nu\E\int_0^t\|\grad z^\nu\|_{L^2(D)}^2.
    \]
    This is done through the use of the trace theorem to bound the viscous dissipation below by the behavior of the flow at the wall. Hence if $v^\nu$ is compactly supported within the bulk/interior of the domain then a similar argument shows that $u^\nu$ will also exhibit anomalous dissipation. We note that while we were unable to show that $v^\nu$ is compactly supported within the interior of the fluid, and thus cannot extend our result to show the existence of solutions to the Navier Stokes problem which satisfy \eqref{eqn: anomalous dissipation}, this kind of assumption is well supported by experimental evidence \cite{pirozzoli2010dynamical} and the references within.
\end{remark}

\section{Existence of (Global) Anomalous Dissipation}\label{sec: exist anomalous dissipation}
Now we will choose specific choices for $f,g$ so that the family of (viscous) weak solutions $\{z^\nu\}$ to the Stokes problem \eqref{eqn: general linear stokes problem} satisfy \eqref{eqn: anomalous dissipation} by blowing up at the boundary. While it may be possible to generalize the arguments used here to arbitrary $C^{1,1}$ or even Lipschitz domains, for clarity we will restrict ourselves to the case when $D = B(0, R) \subset \R^3$ to make the computation along the boundary easier to work with. 
Furthermore, we expect a similar behavior to be true for the full Navier Stokes solution $\{u^\nu\}$, however our approach is not immediately amendable to this question due to possible contributions $v^\nu$ may make at the boundary. See \ref{remark: navier stokes anomalous dissipation}. 

Before getting into all of the details, let us naively sketch out our approach.
\begin{remark}\label{remark: Naive approach}
    Let $0 < \varepsilon \ll 1$. Since $D \in C^1$ there exists a trace operator $\gamma: H^{\frac{1}{2} + \varepsilon}(D) \to L^2(\partial D)$ \cite{schneider2010trace}. Moreover $H^{\frac{1}{2} + \varepsilon}(D)$ is an interpolation space for $L^2(D)$ and $H^1(D)$. Thus by Holder's inequality 
    \begin{align*}
         \nu^{\frac{1}{2} + \varepsilon}\E\int_0^T\|z^\nu\|_{L^2(\partial D)}^2 &\leq \nu^{\frac{1}{2} + \varepsilon}C\E\int_0^T\|z^\nu\|_{H^{\frac{1}{2} + \varepsilon}(D)}^2\\
         &\leq C\Big(\E\int_0^T\|z^\nu\|_{L^2(D)}^2\Big)^{1/2 - \varepsilon}\Big(\nu\E\int_0^T\|z^\nu\|_{H^1(D)}^2\Big)^{1/2 + \varepsilon}.
    \end{align*}
    Recall that when $f,g \in [L^2(D)]^3$ then the kinetic energy within the interior of the domain is uniformly bounded due to \eqref{eqn: Stokes energy uniform bounds}. Hence if $z^\nu$ blows up at the boundary then its $H^1$ norm also blows up at the same rate. Note that there is no contradiction in requiring $z^\nu$ to have a uniform bound on its bulk kinetic energy (i.e. kinetic energy over the interior) and requiring the kinetic energy at the boundary to blow up since the boundary $\partial D$ is a set of 3-dimensional Lebesgue measure 0.
\end{remark}

While this naive approach seems great at first, its impossible to construct a solution $z^\nu$ which blows up faster than $\nu^{\frac{1}{2}}$ at the boundary.
\begin{remark}\label{remark: boundary convergence}
   Suppose $a \in (\frac{1}{2}, 1)$. Now select $\varepsilon = \frac{2a - 1}{4} \in (0,\frac{1}{4})$. Then the linear trace operator $\gamma: H^{\frac{1}{2}+\varepsilon}(D) \to L^2(\partial D)$ is uniformly bounded and by the same argument as in Remark \ref{remark: Naive approach}
    \[
        \nu^a\E\int_0^T\|z^\nu\|_{L^2(\partial D)}^2\;dt \leq  C \nu^{a - \varepsilon - \frac{1}{2}}(\E\int_0^T\|z^\nu\|_{L^2(D)}^2)^{\frac{1}{2} - \varepsilon}(\nu\E\int_0^T\|z^\nu\|_{H^1(D)}^2)^{\frac{1}{2} + \varepsilon} \leq C\nu^{a-\varepsilon - \frac{1}{2}}.
    \]
    Here the last inequality comes from the uniform bounds on the mean kinetic energy and viscous dissipation of $z^\nu$ from \eqref{eqn: Stokes energy uniform bounds}.
    Hence the right hand side vanishes as $\nu \to 0$. As such our naive approach to the problem which we outlined in Remark \ref{remark: Naive approach} needs to be adjusted slightly.
\end{remark}

Since the Stokes problem \ref{eqn: Stokes w/ slip} is linear, to examine how $z^\nu$ behaves at the wall it is enough to study how the noise/ external forcing interacts with the wall to introduce vorticity into the fluid. We now explicitly construct an example of external forcing which is in $[L_\sigma^2(D)]^3$ and blows up along the boundary when convolved with a Gaussian kernel. 
\begin{proposition}\label{proposition: choice of forcing}
    Let $\delta \in (0,1)$, $R > 0$, $D = B(0,R)$, and $\{\mathbf{e}_j\}_{j=1}^3$ be the standard basis in $\R^3$. Let $\dist(x, A) = \ds\inf_{y \in A}\|x-y\|_{\ell^2}$ be the Euclidean distance to set the $A$. 
    Define
    \[
        g(x) := \frac{1}{\dist(x, \partial D)^{\delta/2}}\frac{-\sqrt{x_1^2 + x_2^2}\mathrm{e}_1 + x_3\mathrm{e}_2}{|x|} \quad x \in D.
    \]
    Then $g \in [L^2(0,T, L_\sigma^2(D))]^3$.
    Furthermore, for all $c,b,w > 0$ there exists a constant $C_\delta = C(c,b,\delta) >  0$ such that
    \begin{align*}
        \liminf_{\nu \to 0}\nu^{\delta/2}\E\int_0^T\int_{\partial D}\Big|\int_0^t\int_D \frac{c}{(\nu (t-s))^{3/2}}e^{-b\frac{|x-y|^2}{\nu (t-s)}}e^{-w\nu (t-s)}&g(y)dydW_s\Big|^2\H^2(dx)dt\\
        &\geq C_\delta T^{2-\delta/2} > 0.
    \end{align*}
\end{proposition}
\begin{proof}
    Notice that for $x \in B(0,R)$ then $\dist(x, \partial D) = R - |x|$ and $\widehat{\phi}(x) = \frac{-\sqrt{x_1^2 + x_2^2}\mathrm{e}_1 + x_3\mathrm{e}_2}{|x|}$ is the azimuthal unit vector. As such we can rewrite $g$ is spherical coordinates as 
    \[
        g(x) = \frac{1}{(R-r)^\delta}\widehat{\phi}
    \]
    where $|x| = r$. Then using spherical coordinates and the change of variables $R - r \mapsto r$ provides
    \begin{align*}
        \int_0^T\int_D|g|^2\;dxdt &= \int_0^T\int_0^R \int_{S^2}\frac{r^2}{(R - r)^\delta}\;dS(n)drdt\\
        &\leq 4\pi T\int_0^R \frac{(R-r)^2}{r^\delta}dr < \infty
    \end{align*}
    and by applying the definition of divergence in spherical coordinates we can show that $g$ is divergence free:
    \[
        \grad \cdot g = \frac{1}{r\sin\theta}\pd{}{\phi}\Big(\frac{1}{(R - r)^{\delta/2}}\Big)= 0. 
    \]
    
    Let $\chi_D(x)$ be the indicator function for the set $D$ and let $x \in \partial D$. Consider the change of variables $\xi = \frac{x-y}{\sqrt{\nu t}}$. Note the Triangle inequality implies that for all $h \in \R^3$
    \[
        |x-h| \leq |x| + |h| = R + |h| \Rightarrow |g(x-h)|^2 \geq \frac{1}{|h|^\delta}.
    \]
    Thus by applying Ito's isometry and our lower bound on $g(x-h)$ we obtain
    \begin{align*}
        S_\nu(t) &:= \nu^{\delta/2}\E\big|\int_0^t\int_D \frac{c}{(\nu (t-s))^{3/2}}e^{-b\frac{|x-y|^2}{\nu (t-s)}}e^{w\nu (t-s)}g(y)dydW_s\big|^2\\
        &= c^2\nu^{\delta/2} \E\big|\int_0^t\int_{\R^3} e^{-b|\xi|^2}e^{-w\nu (t-s)}g(x-\xi\sqrt{\nu (t-s)})\chi_D(x-\xi\sqrt{\nu (t-s)})\;d\xi dW_s\big|^2\\
        &=  c^2\nu^{\delta/2} \int_0^t\int_{\R^3} |e^{-b|\xi|^2}e^{-w\nu (t-s)}g(x-\xi\sqrt{\nu (t-s)})|^2\chi_D(x-\xi\sqrt{\nu (t-s)})\;d\xi ds\\
        &\geq c^2\nu^{\delta/2}\int_0^t\int_{\R^3} \frac{e^{-2b|\xi|^2}e^{-2w\nu (t-s)}}{|\xi|^\delta (\nu (t-s))^{\delta/2}}\chi_D(x-\xi\sqrt{\nu (t-s)})\;d\xi ds.
    \end{align*}
    Since $D = B(0, R)$ is convex, the radial lines $x - \xi\sqrt{\nu (t-s)}$ remain inside $D$ for $\nu$ sufficiently small whenever the angle between the vectors $x$ and $\xi$ is between $\frac{\pi}{2}$ and $\frac{3\pi}{2}$ radians. In other words if $\langle \cdot, \cdot \rangle_2$ is the $\ell^2$ inner product in $\R^3$ (i.e. the standard dot product in $\R^3$) then $x - \xi\sqrt{\nu (t-s)} \in D$ for some $\nu$ sufficiently small whenever $\langle x, \xi\rangle_2 < 0$. As such, the indicator functions $\chi_D(x - \xi\sqrt{\nu (t-s)}) \to \chi_{\langle \xi, x\rangle_2 < 0}$ in the sense of distributions as $\nu \to 0$ and for any fixed $x \in \partial D$, in the inviscid limit, the indicator function is supported over exactly half of the unit sphere (in $\xi$) relative to $x$.
    Hence by applying Fatou's Lemma and writing the $\xi$ integral in spherical coordinates, we obtain
    \begin{align*}
        \liminf_{\nu \to 0} \int_0^T\int_{\partial D} &S_\nu(t)\;\H^2(dx) dt\\
        &\geq \int_0^T\int_{\partial D} \liminf_{\nu \to 0}S_\nu(t)\;\H^2(dx)dt\\
        &= c^2 \int_0^T\int_0^t (t-s)^{-\delta/2}dsdt\int_{S^2} \int_{\{\xi \in \R^3 \mid \langle \xi, x\rangle_2 < 0\}} \frac{e^{-2b|\xi|^2}}{|\xi|^\delta}\;d\xi dS(x)\\
        &= \frac{c^2T^{2-\delta/2}}{(1-\delta/2)(2-\delta/2)}\int_{S^2}\int_{\{y \in S^2 \mid \langle y, x\rangle_2 < 0\}}\int_0^\infty e^{-2br^2}r^{2-\delta}\;drdS(y)dS(x)\\
        &= \frac{2\pi^2c^2T^{2-\delta/2}}{(2b)^{(3-\delta)/2}(1-\delta/2)(2-\delta/2)}\Gamma(3-\delta)\\
        &= C_\delta T^{2-\delta/2} > 0.
    \end{align*}
\end{proof}

\begin{proposition}
    For $\delta \in (0,1)$, let $p\delta < 1$ and $s \in (0,1)$ such that 
    \begin{equation}\label{ineq: positive exponent}
         \delta - 2s + 3 - \frac{6}{p} > 0
    \end{equation}
    then $g \in [L^2(0,T, H^s(D))]^3$. 
\end{proposition}
\begin{proof}
    By the definition of $\widehat{\phi}$, the Triangle inequality, and Young's inequality
    \begin{align*}
        |\widehat{\phi}(x) - \widehat{\phi}(y)|^2 &= \Big(\frac{\sqrt{x_1^2 + x_2^2}}{|x|} - \frac{\sqrt{y_1^2 + y_2^2}}{|y|}\Big)^2 + \Big(\frac{x_3}{|x|} - \frac{y_3}{|y|}\Big)^2\\
        &\leq \frac{2|y|^2|x-y|^2 + 2|y|^2|x-y|^2}{|x|^2|y|^2}\\
        &= \frac{4|x-y|^2}{|x|^2}
    \end{align*}
    Then using the Triangle (and Reverse Triangle) inequality, we have that for all $x, y \in D$
    \begin{align*}
        |g(x) - g(y)| &= \Big|\frac{\widehat{\phi}(x)}{(R-|x|)^{\delta/2}} - \frac{\widehat{\phi}(y)}{(R-|y|)^{\delta/2}}\Big|\\
        &= \Big|\frac{1}{(R-|x|)^{\delta/2}} - \frac{1}{(R-|y|)^{\delta/2}}\Big| + \frac{|\widehat{\phi}(x) - \widehat{\phi}(y)|}{(R-|y|)^{\delta/2}}\\
        &\leq \frac{|x-y|^{\delta/2}}{(R-|x|)^{\delta/2}(R-|y|)^{\delta/2}} + \frac{2|x-y|}{|x|(R-|y|)^{\delta/2}}.
    \end{align*}
    Moreover, for $x,y \in D$ the max difference $|x-y| \leq 2R$ and we can extend $g$ to $\R^3$ as 
    \[
        \ds\tilde{g}(x) = \begin{cases}
            \frac{1}{(R-|x|)^{\delta/2}} & |x| < R\\
             0 & |x| > R
        \end{cases}.
    \]
    Putting this all together allows us to bound the $H^s$ semi-norm by
    \begin{align*}
        \int_D \int_D \frac{|g(x) - g(y)|^2}{|x-y|^{3+2s}}dxdy &\leq \int_D\int_D \frac{|x-y|^\delta}{|x-y|^{3+2s}(R-|x|)^{\delta}(R-|y|)^\delta}\;dydx \\
        &\quad + \int_D\int_D \frac{4}{(R-|y|)^\delta |x|^2|x-y|^{1+2s}}\;dxdy\\
        &\leq  \int_D \frac{1}{(R - |x|)^\delta}\int_{|x-y| \leq R}  \frac{|x-y|^\delta}{|x-y|^{3+2s}}\tilde{g}(y)^2\;dydx\\
        &\quad +  \int_D \frac{1}{(R - |x|)^\delta}\int_{R \leq |x-y| \leq 2R}  \frac{|x-y|^\delta}{|x-y|^{3+2s}}\tilde{g}(y)^2\;dydx\\
        &\quad + \int_D\int_D \frac{4}{(R-|y|)^\delta |x|^2|x-y|^{1+2s}}\;dxdy\\
        &= I_1 + I_2 + I_3.
    \end{align*}
    
    As $\tilde{g} \in L^2(\R^3)$  it follows quickly that for any $s > 0$ and $\delta \in (0,1)$
    \[
        I_2 \leq \int_D \frac{1}{(R - |x|)^\delta}\int_{R \leq |x-y| \leq 2R}  \frac{(2R)^\delta}{R^{3+2s}}\tilde{g}(y)^2\;dydx \leq 2^\delta R^{\delta - 2s-3}\|g\|_{L^2(D)}^2 < \infty.
    \]
    Next to study $I_1$ we use a technique from \cite{hedberg1972certain} to write the inner integral as a convolution which is averaged over a ball of radius $R$. 
    Let $p > 1$ be chosen such that $p\delta < 1$ and let $q\geq 1$ be its Holder conjugate, i.e. $\frac{1}{p} + \frac{1}{q} = 1$.  Then  by Young's Convolution inequality for all $r > 0$
    \[
        \|r^{-3} \chi_{B(0, r)} * \tilde{g}^2\|_{L^q(\R^3)} \leq r^{-3}\|\chi_{B(0,r)}\|_{L^w(\R^3)}\|\tilde{g}^2\|_{L^p(\R^3)} = r^{\frac{3-3w}{w}}(\frac{4\pi}{3})^{1/w}\|g^2\|_{L^p(D)}.
    \]
    Here 
    \[
        1 + \frac{1}{q} = \frac{1}{p} + \frac{1}{w} = 1 - \frac{1}{q} + \frac{1}{w} \Rightarrow w = \frac{q}{2}
    \]
    and when $p\delta < 1$
    \[
        \|g^2\|_{L^p(D)}^p = \int_D |g|^{2p}\;dx = \int_{S^2}\int_0^R \frac{r^2}{(R-r)^{p\delta}}drdS(n) \leq 4\pi R^2 \int_0^R u^{-p\delta} du < \infty.
    \]
    Thus by Holder's inequality
    \begin{align*}
        I_1 &= \int_D \frac{1}{(R-|x|)^\delta}\sum_{n=0}^\infty \int_{R2^{-n-1} \leq |x-y| \leq R2^{-n}}\frac{|x-y|^\delta}{|x-y|^{3+2s}}\tilde{g}^2(y)\;dydx\\
        &\leq \int_D \frac{1}{(R-|x|)^\delta}\sum_{n=0}^\infty \frac{(R2^{-n})^\delta}{(R2^{-n - 1})^{3+2s}}\int_{|x-y| \leq R2^{-n}}\tilde{g}^2(y)\;dydx\\
        &\leq 8\sum_{n=0}^\infty (R2^{-n})^{\delta - 2s} (R2^{-n})^{-3}\int_{D} g^2(x) \chi_{B(0, R2^{-n})}*\tilde{g}^2(x)\;dx \\
        &\leq C\|g^2\|_{L^p(D)}^2\sum_{n=0}^\infty (R2^{-n})^{\delta - 2s + \frac{6-3q}{q}}.
    \end{align*}
    To ensure that this geometric series converges we choose $\delta, s, p,q$ such that $p\delta < 1$ and
    \[
        \delta - 2s + \frac{6-3q}{q} = \delta - 2s + \frac{3p-6}{p}> 0.
    \]
    Lastly, we consider the $I_3$ integral due to differences in the azimuthal angle. It follows by Young's convolution inequality that provided $t \geq 1$ satisfies
    \[
        \frac{1}{p} + \frac{3}{4} + \frac{1}{t} = 2 \quad \text{ and } \quad t(1+2s) < 3
    \]
    then
    \begin{align*}
        I_3 &= \int_D\int_D \frac{1}{(R-|y|)^\delta}\frac{1}{|x-y|^{1+2s}}\frac{1}{|x|^2}\;dxdy\\
        &\leq \Big(\int_D\frac{1}{(R-|y|)^{p\delta}}dy\Big)^{1/p}\Big(\int_D \frac{1}{|y|^{8/3}}dy\Big)^{3/4}\Big(\int_D\frac{1}{|y|^{t(1+2s)}}\;dy)^{1/t} < \infty.
    \end{align*}
    Importantly we note that it is always possible to choose $t$ such that this is true. We see that by substituting in our relationship for $t$ into \eqref{ineq: positive exponent}  results in 
    \[
        0 < \delta - 2s + 3 - \frac{6}{p} = \delta - 2s - 9 + 6(2-\frac{1}{p}) = \delta -2s - \frac{9}{2} + \frac{6}{t}
    \]
    and thus 
    \[ 
        t(1+2s) < t(\delta - \frac{7}{2}) + 6 < 3 \Rightarrow t > \max\{1, \frac{3}{\frac{7}{2} - \delta}\}.
    \]
    The supremum (in $\delta$) of which is $\frac{6}{5}$, so for every value of $\delta$ we can choose $t$ sufficiently small so that $t(1+2s) < 3$ which implies $I_3$ is finite.
\end{proof}
\begin{remark}\label{remark: Hs regularity of g}
    Let $0 < \varepsilon \ll \frac{1}{2}$ and define $s = \frac{1-\delta + 2\varepsilon}{2-\delta}$ then \eqref{ineq: positive exponent} is satisfied for all $\delta < \frac{11- \sqrt{41}}{10} \approx 0.45969$.

    To see why, note that by substituting in for $s$ we make a common dominator for the fractions to get
    \begin{align*}
        \delta - \frac{2(1-\delta + 2\varepsilon)}{2-\delta} + 3 - \frac{6}{p} &=
        \frac{p(\delta - \delta^2 + 4 - 4\varepsilon) + 6\delta - 12}{p(2-\delta)} > 0\\
        &\Rightarrow p > \frac{12 - 6\delta}{4 + \delta - \delta^2 - 4\varepsilon}.
    \end{align*}
    Also, it is assumed that $p\delta < 1$ so 
    \[
        \frac{12 - 6\delta}{4 + \delta - \delta^2 - 4\varepsilon} < p < \frac{1}{\delta} \Rightarrow 0 < 4(1-\varepsilon)-11\delta + 5\delta^2.
    \]
    Which is satisfied for $\delta < \frac{11 - \sqrt{121 - 80(1-\varepsilon)}}{10} < \frac{11 - \sqrt{41}}{10}\approx 0.45969$
\end{remark}

Now that we have a possible candidate for our noise coloring, let us show that $z^\nu$ inherits this behavior at the wall. To avoid confusion, we keep $g$ according to Proposition \ref{proposition: choice of forcing} but for simplicity we take $z_0, f \equiv 0$. 
\begin{theorem}\label{thm: exist blow up along boundary}
    Let $g$ be as in Proposition \ref{proposition: choice of forcing} and assume $z_0, f \equiv 0$. Then there exists a kernel $K_t\in [L^\infty(\overline{D} \times \overline{D})]^{3 \times 3}$ for all $t > 0$ such that
    \[
        z^\nu(x,t) = \int_0^t\int_D K_{\nu(t-s)}(x,y)g(y)\;dydW_s
    \]
    is the unique weak solution to the linear Stokes problem  \eqref{eqn: general linear stokes problem} and 
    \begin{align*}
        \liminf_{\nu \to 0} \nu^{\delta/2}\E\int_0^T\|z^\nu\|_{L^2(\partial D)}^2\;dt > 0.
    \end{align*}
\end{theorem}
\begin{proof}
    Recall that $B(0,R)$ is of class $C^1$ and $\alpha \in L^\infty(\partial D)$. Moreover the Stokes operator $A_{NS} = \P(-\Delta)$ is associated with the sesquilinear form $a_{\alpha, \nu}: H_{\sigma, \tau}^1(D) \times H_{\sigma, \tau}^1(D) \to \R$ 
    \[
        a_{\alpha, \nu}(z, \phi) = \nu \int_{\partial D} \alpha z \cdot \phi \;\H^2(dx) + \nu \int_{D}\grad z:\grad \phi\;dx - \int_D \phi \cdot gdW_t(x).
    \]
    Thus the semigroup generated by $A_{NS}$: $e^{-tA_{NS}}$, has a kernel $K_t \in [L^\infty(\overline{D} \times \overline{D})]^3$ for all $t > 0$ \cite{ter2020holder}. As such the solution $z^\nu$ can be represented for all $x \in \overline{D}$ and $t > 0$ as
    \[
         z^\nu(x,t) = \int_0^t e^{-\nu(t-s)A_{NS}}g\;dW_s = \int_0^t\int_D K_{\nu(t-s)}(x,y)g(y)\;dW_s(y).
    \]
    Moreover, as $A_{NS}$ is a self-adjoint operator, $K_t$ also satisfies the lower Gaussian bound \cite{ter2020holder}: for all $x,y \in \overline{D}$ and $t > 0$ there exists $b, c, w > 0$ such that 
    \[
        K_t(x,y) \geq ct^{-3/2}e^{-b\frac{|x-y|^2}{t}}e^{-wt} =: H_t(x,y).
    \]
    Then as $g \geq 0$ and $K_t, H_t \geq 0$ 
    \[
        \E\Big|\int_0^t\int_D K_{\nu(t-s)}(x,y)g(y)dydW_s\Big|^2 \geq \E\Big|\int_0^t\int_D H_{\nu(t-s)}(x,y)g(y)dydW_s\Big|^2.
    \]
    Therefore, by Proposition \ref{proposition: choice of forcing}
    \begin{align*}
        \liminf_{\nu \to 0}\nu^{\delta/2}\E\int_0^T\|z^\nu\|_{L^2(\partial D)}^2
        &= \liminf_{\nu \to 0}\nu^{\delta/2} \E\int_0^T\int_{\partial D}\Big|\int_0^t\int_D K_{\nu (t-s)}(x,y)g(y)\;dydW_s\Big|^2d\H^2(x)dt\\
        &\geq \liminf_{\nu \to 0}\nu^{\delta/2}\E\int_0^T\int_{\partial D}\Big|\int_0^t\int_D H_{\nu(t-s)}(x,y)g(y)\;dydW_s\Big|^2d\H^2(x)dt\\
        &\geq C_\delta T^{2-\delta/2} > 0.
    \end{align*}
\end{proof}

\begin{remark}\label{remark: Besov Space Interpolation}
While we have shown that the kinetic energy at the wall can be made to blow up like $\nu^{\delta/2}$ we cannot apply the same argument as in Remark \ref{remark: Naive approach} since its blow up rate is too slow for the interpolation between $L^2(D)$ and $H^1(D)$. Instead we amend this approach to instead interpolate between $H^s$ and $H^1$ for an appropriate choice of $s$.

\end{remark}

\begin{proposition}\label{proposition: z Hgamma regularity}
    Let $z_\delta^\nu$ be the solution to the Stokes problem under the same assumptions as in Theorem \ref{thm: exist blow up along boundary}. Let $\delta < \frac{11-\sqrt{41}}{10}$ and $0 < \varepsilon \ll \frac{\delta}{4}$ then $\gamma = \frac{1-\delta + 2\varepsilon}{2-\delta} \in (\frac{1-\delta}{2-\delta}, \frac{1}{2})$ and
    \[
        \sup_{\nu \in (0, 1)} \E\int_0^T\|z_\delta^\nu\|_{H^\gamma(D)} < \infty.
    \]
\end{proposition}
\begin{proof}
    By Ito's Isometry
    \[
        \E\int_0^T\|z_\delta^\nu\|_{H^\gamma(D)}^2 = \E\int_0^T \int_0^t\|e^{-\nu(t-s)A_{NS}}g\|_{H^\gamma(D)}^2\;dsdt.
    \]
    Then since $B(0,R)$ is $C^\infty$, it follows that the Stokes semigroup $e^{-\nu tA_{NS}}$ is $C_0$ \cite{ter2020holder} and thus there exists constants $\omega_0 \geq 0$ and $M \geq 1$ such that $\|e^{-\nu tA_{NS}}\| \leq Me^{\nu \omega_0 t}$. As $T < \infty$  
    it follows from Remark \ref{remark: Hs regularity of g} that 
    \[
        \sup_{\nu \in (0,1)} \E\int_0^T\|z_\delta^\nu\|_{H^\gamma(D)}^2 \leq \sup_{\nu \in (0,1)}M\E\int_0^T \int_0^t e^{\omega_0(t-s)}\|g\|_{H^\gamma(D)}^2\;dsdt < \infty.
    \]
\end{proof}

\begin{theorem}\label{thm: exist anomalous dissipation}
    Under the same assumpations as Theorem \ref{thm: exist blow up along boundary} with $\delta < \frac{11 - \sqrt{41}}{10}$, if $z_{\delta}^\nu$ is the weak solution to \eqref{eqn: general linear stokes problem}, then 
    \[
        \limsup_{\nu \to 0}\nu\E\int_0^T\|\grad z_{\delta}^\nu\|_{L^2(D)}^2 > 0.
    \]
\end{theorem}
\begin{proof}
    \textbf{Step 1:} Let $0 < \varepsilon \ll \frac{\delta}{2}$. Define $s_\delta = \frac{1-\delta + 2\varepsilon}{2-\delta}$. By Proposition \ref{proposition: z Hgamma regularity}
    \[
        \sup_{\nu \in (0,1)}\E\int_0^T\|z_{\delta}^\nu\|_{H^{s_\delta}(D)}^2 < \infty.
    \]
    \textbf{Step 2:} As $D = B(0,R)$ is $C^1$, there exists a bounded linear trace operator $\mathrm{tr}: [H^{1/2 + \varepsilon}(D)]^3 \to [L^2(\partial D)]^3$ \cites{schneider2011traces, adams2003sobolev}. Moreover, the Sobolev space $H^{1/2 + \varepsilon}(D)$ is an interpolation space for $H^{s_\delta}(D)$ and $H^1(D)$ 
    with interpolation exponent $\theta \in (0,1)$ given by  \cites{adams2003sobolev, triebel1973spaces}
    \[
        \frac{1}{2} + \varepsilon = (1-\theta)s_\delta + \theta \Rightarrow \theta = \frac{\delta}{2}.
    \]
    
    Therefore, after applying Holder's inequality with conjugates $p = \frac{2}{\delta}$ and $q = \frac{2}{2-\delta}$ we get
    \begin{align}\label{eqn: interpolation}
        \E\int_0^T\|z_{\delta}^\nu\|_{L^2(\partial D)}^2 &\leq C\E\int_0^T\|z_{\delta}^\nu\|_{H^{1/2 + \varepsilon}(D)}^2\nonumber\\
        &\leq C\Big(\E\int_0^T\|z_{\delta}^\nu\|_{H^{s_\delta}(D)}^2\Big)^{\frac{2-\delta}{2}}\Big(\E\int_0^T\|z_{\delta}^\nu\|_{H^1(D)}^2\Big)^{\frac{\delta}{2}}.
    \end{align}
    \textbf{Step 3:}
    It then follows from Theorem \ref{thm: exist blow up along boundary} and \eqref{eqn: interpolation} that
    \begin{align*}
        0 &< \limsup_{\nu \to 0}\nu^{\frac{\delta}{2}}\E\int_0^T\|z_\delta^\nu\|_{L^2(\partial D)}^2\;dt\\
        &\leq \limsup_{\nu \to 0}C\Big(\E\int_0^T\|z_{\delta}^\nu\|_{H^{s_\delta}(D)}^2\Big)^{\frac{2-\delta}{2}}\Big(\nu\E\int_0^T\|z_{\delta}^\nu\|_{H^1(D)}^2\Big)^{\frac{\delta}{2}}\\
        &\leq C \limsup_{\nu \to 0}\Big(\nu\E\int_0^T\|\grad z_{\delta}^\nu\|_{L^2(D)}^2\;dt\Big)^{\frac{\delta}{2}}.
    \end{align*}
    Here the last inequality is due to the uniform bound on the kinetic energy in \eqref{eqn: Stokes energy uniform bounds} so
    \begin{align*}
        \limsup_{\nu \to 0}\nu\E\int_0^T\|z_{\delta}\|_{H^1(D)}^2 &=  \limsup_{\nu \to 0}\nu\E\int_0^T\|z_{\delta}\|_{L^2(D)}^2 +  \limsup_{\nu \to 0}\nu\E\int_0^T\|\grad z_{\delta}\|_{L^2(D)}^2\\
        &= \limsup_{\nu \to 0}\nu\E\int_0^T\|\grad z_{\delta}\|_{L^2(D)}^2.
    \end{align*}
\end{proof}

\begin{remark}
    Here we were able to construct a linear problem with a non-vanishing viscous dissipation term by concentrating the kinetic energy at the wall.  This shows that even for non-statistically stationary solutions \eqref{eqn: anomalous dissipation} is not a good measure of how the nonlinearity contributes to the increase in small scale oscillations and accelerates viscous dissipation within a turbulent flow. In part, this is because \eqref{eqn: anomalous dissipation} only measures  the total enstrophy in the system and not how the convective acceleration speeds up dissipation within the interior \cite{chen2010study}. As such if we choose our forcing appropriately we can create a sufficient amount of vorticity at the boundary so that the average total enstrophy in the system blows up at its max possible rate of $O(\nu^{-1})$. 
\end{remark}
\begin{remark}
   Previously Bedrossian et.al. \cite{bedrossian2019sufficient} showed that 
   global anomalous dissipation could be achieved by statistically stationary solutions to the heat equation where nothing ``anomalous" is actually occurring. This pointed to the fact that \eqref{eqn: anomalous dissipation} is not a good definition for saying a fluid exhibits anomalous dissipation in the context of stochastic flows. Similarly, we have constructed an example of a solution to the linear Stokes problem which satisfies \eqref{eqn: anomalous dissipation} but is not statistically stationary which is a complementary result. However, unlike the result of \cite{bedrossian2019sufficient}, our approach can also be extended to the case of deterministic solutions by choosing $f$ to blow up at the boundary and setting $g \equiv 0$. Hence \eqref{eqn: anomalous dissipation} is not a good definition for saying a fluid exhibits anomalous dissipation if one only requires the external forcing to be $L^2$. 
\end{remark}
\begin{remark}\label{remark: local anomalous dissipation}
   In order to account for the issue of statistically stationary solutions Bedrossian et.al. \cite{bedrossian2019sufficient} suggested the concept of a solution satisfying what they call as weak anomalous dissipation where
   \[   
        \lim_{\nu \to 0}\nu\E\|u^\nu(t)\|_{L^2(D)}^2 = 0 \quad \forall t \in [0, T].
   \]
   However by the uniform bounds on the kinetic energy we see that both $u^\nu$ and $z^\nu$ satisfy this definition as well meaning this does not capture the role of the nonlinearity to accelerate the amount of dissipation that occurs within the interior of the flow. Instead we posit that anomalous dissipation should be measured as a local phenomena instead of a global one. 
   
   One possible local approach would be that instead of using \eqref{eqn: anomalous dissipation}, for a non-statistically stationary fluid to truly exhibit anomalous dissipation it should hold that 
   \[
        \limsup_{\nu \to 0}\nu\E\int_0^T\|\grad u^\nu\|_{L_{loc}^2(D)}^2 > 0.
   \]
   This does not work in the case of statistically stationary solutions due to the heat equation (again) satisfying this bound without anything anomalous occurring. As such we expect that one can similarly construct non-statistically stationary solutions using the techniques here which also satisfy this requirement possibly by using $g$ like we defined here but it concentrates energy on concentric shells with radii from a countable dense set subset of $(0, \infty)$.
   
   Another (local) approach, is to just measure how much the nonlinearity contributes to the dissipation of energy using the regularity measure of Duchon and Robert\cite{duchon2000inertial}:
   \[
        \limsup_{\nu \to 0}\E\int_0^T\int_U D(u^\nu)(dx)dt > 0
   \]
   where 
   \[
        D(u^\nu)(x) = \lim_{\ell \to 0} \frac{1}{4}\int_{D} \grad \phi_\ell(y)\cdot \delta_y u^\nu(x) |\delta_y u^\nu(x)|^2\;dy \quad \forall \; x \in D.
   \]
   Here $\phi_\ell$ is a standard mollifier of size $\ell$ and $\delta_y u^\nu(x) = u^\nu(x+y) - u^\nu(x)$. See \cites{duchon2000inertial, dudley2024necessary} for how $D(u^\nu)$ acts as a measure of the dissipation due to the convective term in the energy balance over a torus. 
\end{remark}
\begin{remark}
    By construction the noise $gdW_t$ is concentrated at the boundary. This is strictly different from Kolmogorov's theory of turbulence which assumes that the external forcing on the system is confined to only the largest length scales of the problem (i.e. within the interior of the domain). Nevertheless one can still analyze the structure functions over the interior such as in \cite{papathanasiou2021sufficient} to examine the impact of the nonlinearity on the energy dynamics. 
\end{remark}
\begin{remark}
    While we have shown that the existence of (global) anomalous dissipation for a system subject to an arbitrary choice of $[L^2(D)]^3$ forcing and initial conditions, is not well defined as a nonlinear phenomena, one could alternatively redefine the problem to be subject to $[L^\infty(D)]^3$ forcing, like in Kolmogorov's theory of turbulence. In this way the forcing will only exist over the largest length scales of the problem and (possibly) cannot be concentrated on sets of Lebesgue measure 0. In this case, the nonlinearity will (possibly) be the driving source of vorticity generation and once again allow \eqref{eqn: anomalous dissipation} to be a good definition of anomalous dissipation. This problem remains open, however we note that vorticity will still be generated at the wall so one will need to account for how this affects the existence of global anomalous dissipation.
\end{remark}

\section{Simulation Examples}\label{sec: numerical simulations}
In order to confirm the results from Section \ref{sec: exist anomalous dissipation} we will simulate the linear Stokes problem over a semi-infinite plate and inside a sphere and measure the amount of viscous dissipation as $\nu \to 0$. Throughout this section we implement a finite difference approach while evolving in time using the Euler-Maruyama method (a forward Euler approach in the context of deterministic flows) for simplicity. In order to ensure that the dissipation is fully resolved we take the spatial step size on the order of the Kolmogorov length scale $dy = \nu^{3/4}$ and temporal step size of $dt = 0.005$ (smaller than the Kolmogorov time scale for all values of $\nu$ we consider here and is compatible with the CFL condition). Finally for all stochastic simulations we average the results over 250 different realizations of the flow.

First we consider the numerically simplest case: an incompressible fluid above an infinite stationary plate contained within the plane $y = 0$. We will also assume that the flow is symmetric with respect to both the $x$ and $z$ axes and decays to 0 as $y \to \infty$. We take our coloring to be $g(x,y,z) = [0, 0, \frac{1}{y^{\delta/2}}]$ with the intial condition of $z_0 \equiv 0$. For now we will fix $\delta = 0.75$. Later we will check to see how the various quantities of interest change with $\delta$. For the actual simulation space we will take the domain to have a height of $y_{max} = 10$, and we fix the slip length $\alpha = 0.0005$. 

\begin{figure}
    \centering
    \begin{subfigure}[b]{0.45\textwidth}
         \centering
         \includegraphics[width=\textwidth]{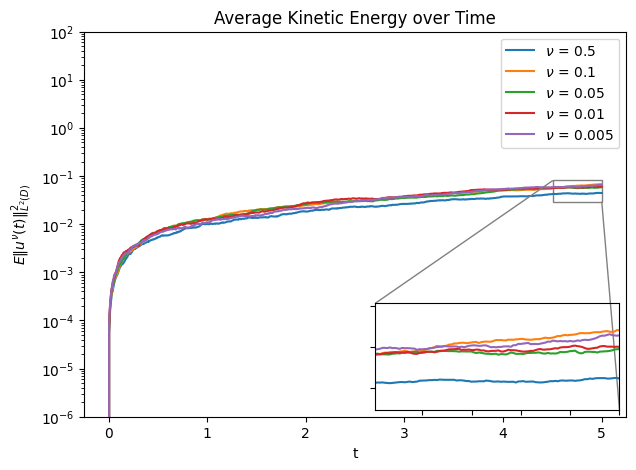}
         \caption{Average Kinetic Energy}
         \label{fig: half space kinetic energy}
     \end{subfigure}
     \hfill
     \begin{subfigure}[b]{0.45\textwidth}
         \centering
         \includegraphics[width=\textwidth]{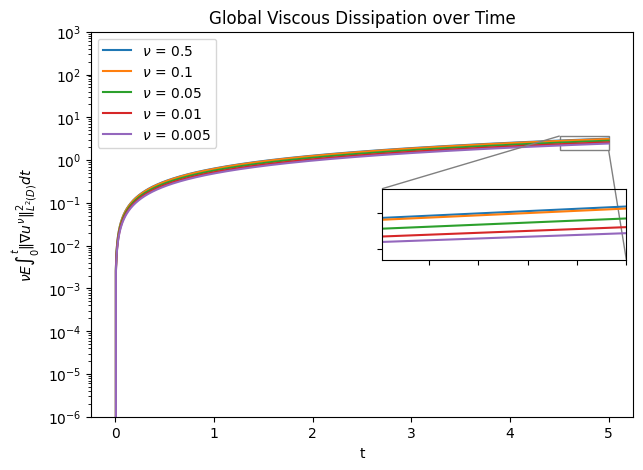}
         \caption{Average Dissipation}
         \label{fig: half space dissipation}
     \end{subfigure}
     \caption{The Average Kinetic Energy and Viscous Dissipation above an infinite plate}
     \label{fig: half space energy and dissipation}
\end{figure}

Due to the symmetry assumption in this setting we can reduce the full 3D problem to a 1D problem which is perpendicular to the plate greatly simplifying the numerical complexity of the problem. In Figure \ref{fig: half space energy and dissipation}, we show that both the mean global kinetic energy and global viscous dissipation remain roughly constant as $\nu \to 0$. This agrees with our earlier analysis that along the boundary the kinetic energy is blowing up as $\nu \to 0$, so even through the energy within the interior of the flow is uniformly bounded with respect to $\nu$, the viscous dissipation does not vanish as vorticity continues to be created from the singularity in the noise at the boundary. See how the mean kinetic energy at the boundary remains roughly constant in Figure \ref{fig: half space kinetic energy at wall}. 

\begin{figure}
    \centering
    \includegraphics[width=0.5\linewidth]{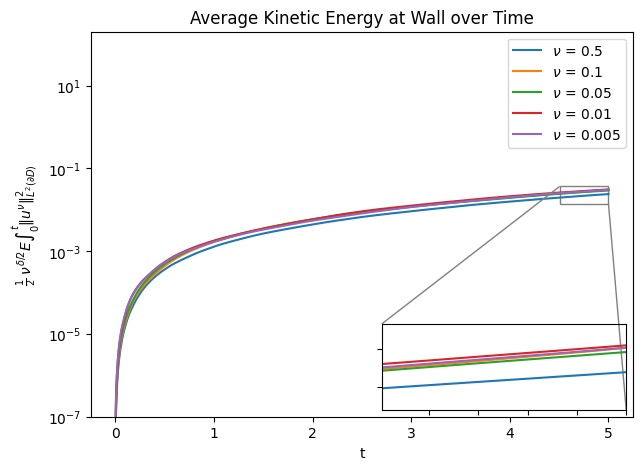}
     \caption{The Average Kinetic Energy at the Plate}
     \label{fig: half space kinetic energy at wall}
\end{figure}

We note that a similar problem to what we have analyzed so far is Stokes' 1st problem where we take a no-slip condition at the boundary and allow the plate to oscillate with fixed speed $U$. This is analogous to the problem we have analyzed as the forcing is concentrated at the boundary and a Brownian motion is normally distributed about 0 so in our case the noise $gdW_t$ may flip signs with every time step. However we note that the creation of anomalous dissipation and the blow up of the kinetic energy are not due to the oscillations but instead due to the singularity of $g$ at the boundary. Indeed, if we instead look at the deterministic system with external forcing still given as $f(x,y,z) =  [0, 0, \frac{1}{y^{\delta/2}}]$ we still see the existence of global anomalous dissipation and a blow up in the kinetic energy at the boundary. See Figure \ref{fig: half space deterministic viscous dissipation}. 
Moreover, looking at the various quantities for the deterministic and stochastic settings, we see that the oscillations at the boundary actually reduce how much energy can accumulate at the boundary and how much vorticity can be introduced into the bulk flow. Most likely this is because in the deterministic setting, the forcing pushes the flow in only one direction, while in the stochastic case the forcing is able to change directions --- meaning that in the deterministic setting the fluid velocity at the boundary increases much faster due to a mono-directional forcing while the oscillatory-like forcing from the stochastic case limits how fast the velocity at the boundary can grow. 

\begin{figure}
    \centering
    \begin{subfigure}[b]{0.45\textwidth}
         \centering
         \includegraphics[width=\textwidth]{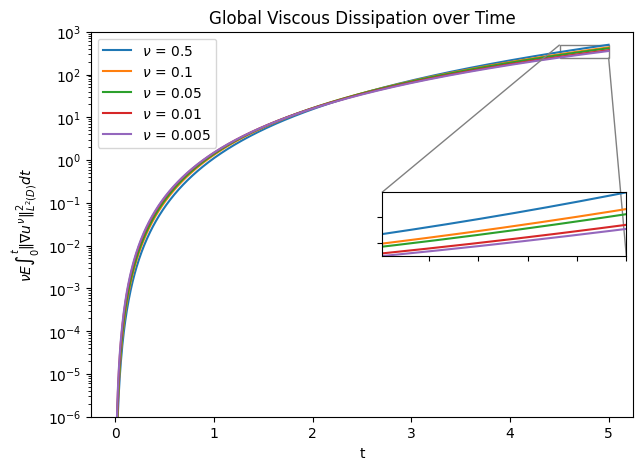}
         \caption{Global Viscous Dissipation}
         \label{fig: half space deterministic dissipation}
     \end{subfigure}
     \hfill
     \begin{subfigure}[b]{0.45\textwidth}
         \centering
         \includegraphics[width=\textwidth]{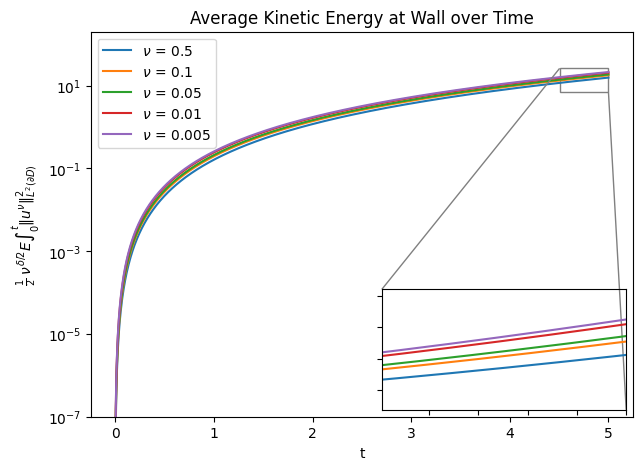}
         \caption{Kinetic Energy at the Wall}
         \label{fig: half space deterministic kinetic energy}
     \end{subfigure}
     \caption{The Total (Global) Viscous Dissipation above an infinite plate in a Deterministic System and Energy at the Wall}
     \label{fig: half space deterministic viscous dissipation}
\end{figure}

So far we have fixed $\delta$ and seen how the quantities of interest change with $\nu$. Moreover the choice of $\delta$ used is outside of the values we used in Theorem \ref{thm: exist anomalous dissipation}, however this is inconsequential as can be seen in Figure \ref{fig: half space dissipation stochastic and deterministic} both the stochastic and deterministic systems behave the same for each value of $\delta$ with only the total amount of viscous dissipation over time increases with the strength of the singularity at the wall. This suggests that Theorem \ref{thm: exist anomalous dissipation} should be able to be extended to all $\delta \in (0,1)$ but the validity of which remains open.

\begin{figure}
    \centering
    \begin{subfigure}[b]{0.45\textwidth}
         \centering
         \includegraphics[width=\textwidth]{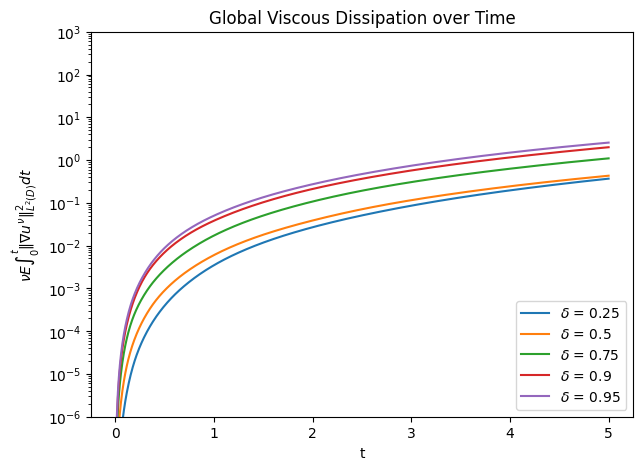}
         \caption{Average Dissipation in Stochastic System}
         \label{fig: half space stochastic dissipation delta}
     \end{subfigure}
     \hfill
     \begin{subfigure}[b]{0.45\textwidth}
         \centering
         \includegraphics[width=\textwidth]{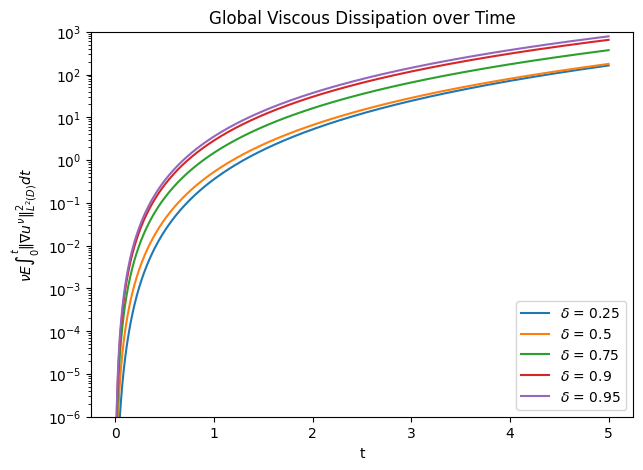}
         \caption{Dissipation in Deterministic System}
         \label{fig: half space deterministic dissipation delta}
     \end{subfigure}
     \caption{The Average Viscous Dissipation above an infinite plate in a Stochastic and a Deterministic System as $\delta \to 1^-$}
     \label{fig: half space dissipation stochastic and deterministic}
\end{figure}

Next we consider the same (Stokes) problem but over a sphere of radius $R = 5$ instead of in a flat semi-infinite domain. We do this first to confirm the results from Theorem \ref{thm: exist anomalous dissipation} as well as to check if the curvature / symmetry assumptions we used in the previous simulation are impacting the results. Note that in this setting the coloring on the noise is given by $g(r, \theta, \phi) = [0, 0, \frac{1}{(5-r)^{\delta/2}}]$. 

\begin{figure}
    \centering
    \begin{subfigure}[b]{0.45\textwidth}
         \centering
         \includegraphics[width=\textwidth]{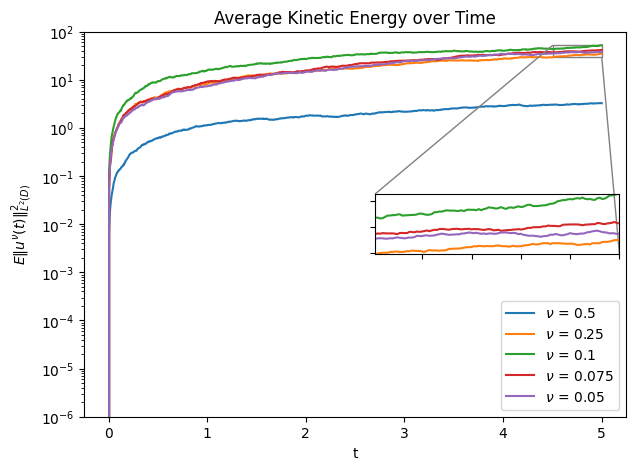}
         \caption{Average Kinetic Energy}
         \label{fig: sphere average kinetic energy}
    \end{subfigure}
    \hfill
    \begin{subfigure}[b]{0.45\textwidth}
         \centering
         \includegraphics[width=\textwidth]{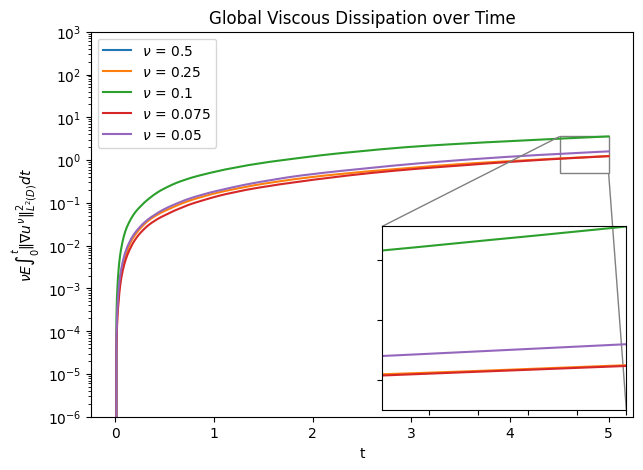}
         \caption{Average Viscous Dissipation}
         \label{fig: sphere average dissipation}
    \end{subfigure}
    \hfill
    \begin{subfigure}[b]{0.45\textwidth}
         \centering
         \includegraphics[width=\textwidth]{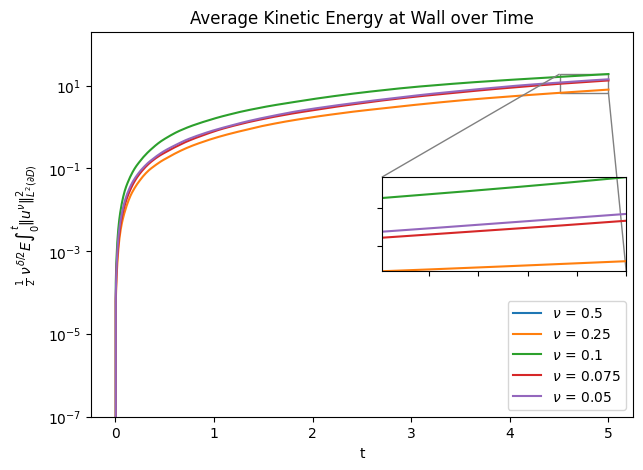}
         \caption{Average Kinetic Energy at the Wall}
         \label{fig: sphere average kinetic energy at the wall}
    \end{subfigure} 
    \caption{The Average Viscous Dissipation and Kinetic Energy over a Sphere of Radius 5}
    \label{fig: Average Dissipation over the Sphere}
\end{figure}

Once again we use a finite difference approach to compute the solution $z^\nu$ but now we take the mesh sizing in the $r, \theta, \phi$ directions to all be on the size of the Kolmogorov scale $\nu^{3/4}$. Since this mesh sizing is in all 3 dimensions, the cost of each time step has greatly increased compared to the halfplane case where the symmetry assumptions reduced the problem to 1 dimension. Moreover we impost the (implicit) assumption used throughout that at the origin $z^\nu$ is bounded. In practice this was done by saying that at the origin $z^\nu$ is equal to the average of $z^\nu$ over the smallest shell containing the origin. 

In order to compensate for the extra computational difficulties, in this case we restrict $\nu$ to a smaller range of values: $\nu = [0.5, 0.25, 0.1, 0.075, 0.05]$. Nevertheless, we still observe that once again the solution to the linear Stokes problem $z^\nu$ exhibits a blow up in the kinetic energy at the wall as well as satisfies the global anomalous dissipation assumption. See Figure \ref{fig: Average Dissipation over the Sphere}. However when comparing the spherical case in Figure \ref{fig: Average Dissipation over the Sphere} with the infinite plate case in \ref{fig: half space energy and dissipation} and \ref{fig: half space kinetic energy at wall} we see that the curvature of the domain causes the kinetic energy at the boundary to blow up much faster. Moreover, it causes the total amount of viscous dissipation and kinetic energy to increase. Most likely this is because in the bounded domain, any oscillations that do not dissipate out near the origin are eventually reflected back at the wall and the curvature makes the newly created oscillations which are moving into the interior to collide with one another leading to the waves amplifying one another.   


\section{Future Directions}
This work has shown that \eqref{eqn: anomalous dissipation} is not a good definition for how nonlinear affects contribute to anomalous dissipation even in the non-statistically stationary/ deterministic setting when the external forcing is taken to be only $L^2$. But it remains open whether requiring the external forcing to be in $L^\infty$ could fix this issue. Another area of interest is to consider the converse: what are sufficient conditions such that solutions to the Navier Stokes equations $u^\nu$ behave well-enough such that
\[
    \limsup_{\nu \to 0}\nu\E\int_0^t\|\grad u^\nu\|_{L^2(D)}^2 = 0.
\]
One possible approach is to examine the techniques used in \cite{drivas2018onsager} for the no-slip boundary condition and extend them to the slip boundary condition case where certain quantities may be easier to work with. Another interesting approach is that Bourgain, Brezis, and Mironescu showed that for Lipschitz domains $D$ and $h \in [H^1(D)]^3$
\[
    \lim_{s \to 1^-}(1-s)\int_D \int_D \frac{|h(x) - h(y)|^2}{|x-y|^{3+2s}}\;dxdy = \frac{2\pi}{3}\|\grad h\|_{L^2(D)}^2
\]
Hence if one can find a condition so that this point-wise limit in the $H^s$ semi-norms is uniform, then using interpolation theory
\[
    \limsup_{\nu \to 0^+}\nu\E\int_0^T\|\grad u^\nu\|_{L^2(D)}^2 = \lim_{s\to 1^-}\limsup_{\nu \to 0^+}(1-s)\nu\E\int_0^T\|u^\nu\|_{\dot{H}^s(D)}^2 = 0.
\]

\section{Acknowledgments}
The authors gratefully acknowledged the support by the National Science Foundation under the grants DMS 2231533 and DMS 2008568. 
Furthermore, the authors thank Dr. Huy Quang Nguyen for his many insightful comments in the drafting of this article.

\bibliographystyle{plain}
\bibliography{paper_references.bib}

\end{document}